\documentclass[11pt,letter]{article}
\usepackage[utf8]{inputenc}
\usepackage[margin=1in]{geometry}
\usepackage{amsfonts,amssymb,amsmath,amsthm,amstext}
\usepackage{dsfont}
\usepackage{hyperref}
\usepackage{url}

\usepackage{bbm,color,tikz}

\newtheorem{theorem}{Theorem}
\newtheorem{lemma}[theorem]{Lemma}
\newtheorem{corollary}[theorem]{Corollary}
\newtheorem{definition}{Definition}

\newtheorem{claim}{Claim}

\newcommand{\inv}[2]{\mathrm{inv}_{#2}(#1)}
\newcommand{\Z}{\mathbb{Z}}
\newcommand{\GF}{\mathbb{F}}
\newcommand{\poly}{\mathcal{P}}

\title{A note on multiplicative-inverse chaining in finite fields}
\author{Divyarthi Mohan\thanks{Boston University. \texttt{dmohan@bu.edu}.} \and R.\ Ravindraraj\thanks{Hawcx Inc. \texttt{ravi@hawcx.com}.}}
\date{}

\begin{document}

\maketitle

\begin{abstract}
    We consider chaining multiplicative-inverse operations in finite fields under alternating polynomial bases. When using two distinct polynomial bases to alternate the inverse operation we obtain a partition of $\GF_{p^n}\setminus \GF_p$ into disjoint cycles of even length. This allows a natural interpretation of the cycles as permutation cycles. Finally, we explore chaining under more than two polynomial bases.
\end{abstract}

\section{Introduction}

Finite field $\GF_{p^n}$ of order $p^n$ for some prime $p$ and positive integer $n>0$ are commonly represented as polynomial quotient rings or vector spaces over $\GF_p$. In particular, the multiplicative inverse is defined with respect to a basis. We explore chaining multiplicative-inverse operations by alternating bases. Using standard polynomial basis representations, we generate a \emph{chain} or sequence of elements/polynomials in $\GF_p[X]$ of degree at most $n-1$ by switching the basis irreducible polynomial between each inverse operation. We make preliminary observations about the chains when alternating the basis between two distinct irreducible polynomials of degree $n$ by considering the graph/matching induced by multiplicative inverse under a given basis. That is, there is an edge between two elements if and only if they are multiplicative inverses of each other. Since for any element with degree greater than $0$ the inverse under two distinct polynomial basis are always distinct~\cite{Conrad_Roots}, the matchings induced by two distinct polynomial basis are strictly disjoint except on elements of the base field $\GF_{p}$. Therefore, alternating between two basis produces an alternating path made of two matchings. This induces a partition of the non-zero degree elements into disjoint cycles of even length, allowing a natural interpretation of the cycles as permutation cycles. That is, chaining two involutions (multiplicative-inverse functions) induces a set of permutations. We note that this is completely different from the composition or product of two involutions, which are well-studied in the literature~\cite{petersen2013write,BurnetteCharles2017FPit}. We also explore chaining inverse operations under more than two bases polynomials. Here the structure of the sequences produced are no longer simple cycles.

\section{Notation and Preliminaries}

First, we briefly describe some notation and standard terminology we use in this paper.  $\GF_q$ denotes the finite field of order $q$. In particular, for any prime $p$, we use $\GF_p = \Z/p\Z$ to denote the finite field of order $p$. %Given any field (or ring) $K$, $K[X]$ denotes the polynomial ring with coefficients in $K$. 
For any $q = p^n$, $\GF_q$ is a field extension of $\GF_p$, and we use \emph{polynomial basis representations} for $\GF_q$. Let $f(X)\in \GF_p[X]$ be an irreducible polynomial of degree $n$, then $\GF_{q}$ is the splitting field of $f$ and is isomorphic to the quotient ring $\GF_p[X]/f(X)$. %Moreover, if $f$ is a \emph{primitive polynomial} then any root of $f(X)$ in $\GF_{q}$, say $\alpha$, generates the multiplicative group $\GF_{q}^*$. 

\paragraph{Polynomial basis representation.} Let $\poly(p,n)$ denote the set of all polynomials in $\GF_p[X]$ of \emph{degree at most} $n-1$. A polynomial $a \in \poly(p,n)$ represents an element of $\GF_{p^n}$, and this representation is well-defined only with respect to a corresponding irreducible polynomial $f(X) \in \GF_p[X]$ of degree $n$. That is, polynomial representations can map to different isomorphisms of $\GF_{p^n}$ and each isomorphism is specified by a unique monic irreducible polynomial of degree $n$.~\footnote{Wlog, throughout the paper, any basis is defined with respect to \emph{monic} irreducible polynomials.} Note that all such isomorphisms always map any zero degree polynomial $a\in \GF_p$ to itself. For any $a_1, a_2\in \GF_p[X]/f(X)$, $a_1 + a_2$ is the standard polynomial addition in $\GF_p[X]$, and the multiplication is defined as $a_1\cdot a_2 \mod f(X)$. In particular, for any polynomial $a_1, a_2 \in \GF_p[X]$ (of degree less than $n$) the field multiplication in $\GF_{p^n}$ is always defined with respect to some irreducible polynomial $f(X) \in \GF_p[X]$ of degree $n$. Further, we use the following notation to denote the multiplicative inverse of an element. %Further, for any polynomial $a\in \GF_p[X]$ of degree less than $n$ we define $\inv(a,f) \in \GF_p[X]$ to be the \emph{multiplicative inverse} with respect to the irreducible polynomial $f(X) \in \GF_p[X]$ of degree $n$. That is, $a\cdot \inv(a,f) \equiv 1 \mod f(X)$.

\begin{definition}[Multiplicative Inverse]
    Let $a\in \GF_p[X]$ be a non-zero polynomial of degree less than $n$. We define $\inv a f \in \GF_p[X]$ to be the \emph{multiplicative inverse} of $a$ with respect to the irreducible polynomial $f(X) \in \GF_p[X]$ of degree $n$. That is, $a\cdot \inv a f \equiv 1 \mod f(X)$.
\end{definition}

%\subsection{Multiplicative inverse, Involution and Matching}
%$\alpha^{2k} = \alpha^{c(q-1)}$ this implies $k = (q-1)$ or $(q-1)/2$.
For any irreducible polynomial $f$ of degree $n$, the multiplicative inverse function $\inv \cdot f$ is an \emph{involution}, i.e., $\inv{\inv a f}{f} = a$ for all $a\in \poly(p,n)\setminus \{0\}$.

\paragraph{Graph representation of $\mathrm{inv}_f$.} It will be helpful to visualize the multiplicative inverse (w.r.t.\ to a given  $f$) as a graph over the elements of $\poly(p,n)\setminus \{0\}$, where two vertices $u,v\in \poly(p,n)$ have an edge between them if and only if $\inv u f = v$. Note that, the induced graph on $\poly(p,n)\setminus \{0,1,p-1\}$ is a perfect matching.

%In other words, for a given irreducible polynomial $f$, the multiplicative inverse function with respect to $f$ induces a matching where the vertices of the graph are the non-constant polynomials of degree at most $n-1$ and $(p_1,p_2) \in E$ iff $\inv(p_1,f) = p_2$.

We next note that, given two distinct irreducible polynomial $f_1,f_2 \in \GF_p[X]$ of degree $n$, the matchings induced by $\mathrm{inv}_{f_1}$ and $\mathrm{inv}_{f_2}$ share a common edge $(u,v)$ if and only if $u,v$ have degree zero (i.e., $u,v \in \GF_p$). We cast this in the following claim.

\begin{claim}\label{claim:distinct inverse}
Let $f_1,f_2 \in \GF_p[X]$ be two irreducible polynomials of degree $n$. For any $a \in \poly(p,n)$ of degree at least $1$, we have $\inv a {f_1} = \inv a {f_2}$ if and only if $f_1  \equiv f_2$.
\end{claim}

\begin{proof} Suppose there exists $f_1 \neq f_2$ such that for some $a,g\in \poly(p,n)$ of degree greater than $0$ we have $\inv a {f_1} = \inv a {f_2} = g$. This implies that there exists $h_1,h_2\in \GF_p[X]$ such that $a(X)\cdot g(X) = 1 + h_1(X)\cdot f_1(X) = 1 + h_2(X)\cdot f_2(X)$, and $h_1,h_2$ have degree at most $n-2$. Hence, we have \[
    h_1(X)\cdot f_1(X) = h_2(X)\cdot f_2(X).\]

In particular, for any root $\alpha$ of $f_1$ we have $h_1(\alpha)\cdot f_1(\alpha) = h_2(\alpha)\cdot f_2(\alpha) = 0$. However, 
any two different~\footnote{Note that we do \emph{not} consider an irreducible polynomial $f(X)$ as being different than $c\cdot f(X)$ for $c\in \GF_p\setminus \{0\}$.} irreducible polynomials over a field do not share any common roots in any field extension~\cite{Conrad_Roots}, hence $f_2(\alpha) \neq 0$. This implies that $h_2(\alpha) = 0$ for all the $n$ distinct roots $\alpha$. But $h_2$ has at most $n-2$ degree, so it must be the case that $h_2 \equiv 0$. This leads to a contradiction since $a$ has degree at least one.

% \begin{align*}
%     \alpha^k &= a \\
%     \alpha^{p-1 - k} &= \inv(a,f_1) \\
%     \inv(\inv(a,f_1),f_2) &\neq a
% \end{align*}
\end{proof}

\section{Inverse Chaining with Two Bases}

In this section we introduce multiplicative inverse chaining by alternating the multiplicative inverse operation with respect to two different irreducible polynomials. In particular, two different polynomial basis representations of $\GF_{p^n}$ are used when considering the multiplicative inverse of an element. We start by defining a \emph{chain} induced by alternating two bases.

\begin{definition}
For any integer $k>0$, polynomial $a\in \poly(p,n)$, and irreducible polynomials $f_1,f_2 \in \GF_p[X]$ of degree $n$, we define the \emph{$k$-chain} of $a$ with respect to $(f_1,f_2)$ as the sequence $(a_0,a_1,\ldots, a_k)$ such that
\[
a_i = 
\begin{cases}
a \quad \text{if } i=0\\
\inv{a_{i-1}}{f_1} \quad \text{if } i \text{ is odd}\\
\inv{a_{i-1}}{f_2} \quad \text{if } i \text{ is even}
\end{cases}
\]
\end{definition}

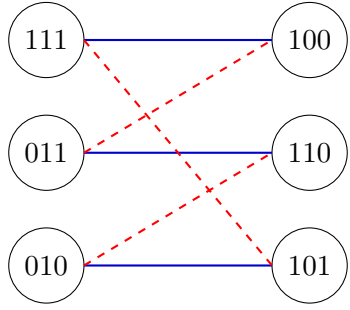
\begin{figure}
\begin{center}
\begin{tikzpicture}

% Define node style
\tikzset{circle node/.style={draw, circle, minimum size=1cm, inner sep=0pt}}

% List of labels
\def\labels{{"111","011","010","100","110","101"}}

% Draw the 4x2 grid (4 rows, 2 columns), 1cm spacing
\foreach \col in {0,1} {
    \foreach \row in {0,1,2} {
        \pgfmathtruncatemacro{\index}{\col*3 + \row}
        \pgfmathsetmacro{\x}{\col * 3.5} % 1.5 cm horizontal spacing
        \pgfmathsetmacro{\y}{- \row * 1.5} % 1.5 cm vertical spacing
        \node[circle node] at (\x,\y) {\pgfmathparse{\labels[\index]}\pgfmathresult};
    }
}

% Connect nodes in each row (horizontal lines)
\foreach \row in {0,1,2} {
    \draw[color = blue!80!black, thick] (0.5,-\row*1.5) -- (3,-\row*1.5);
}

\draw[color = red, thick, dashed] (0.5,-0)-- (3,-3);
\draw[color = red, thick, dashed] (0.5,-3)-- (3,-1.5);
\draw[color = red, thick, dashed] (0.5,-1.5)-- (3,-0);
\end{tikzpicture}
\end{center}
\caption{$\GF_8\setminus \{0,1\}$ under two polynomial basis representations. $b_2b_1b_0$ represents the polynomials $b_2x^2 + b_1x + b_0 \in \GF_2[X]$ of degree at most $2$. The solid blue and dashed red edges represent the inverse operations under irreducible polynomials $x^3+x+1$ and $x^3+x^2+1$ respectively. For example, $(x^2 + x + 1)*(x^2) = 1 \mod x^3+x+1$ so (111){\color{blue!80!black}---}(100), and $(x^2 + x + 1)*(x^2+1) = 1 \mod x^3+x^2+1$ so (111){\color{red}- -}(101).}
\label{fig:complete-chain}
\end{figure}

If $a\in \GF_p$ (i.e., a constant or zero-degree polynomial in $\poly(p,n)$) then the chain is simply an alternating sequence of $a$ and $a^{-1} \in \GF_p$ since $\inv a f = a^{-1}$ for all irreducible polynomial $f$.

Figure~\ref{fig:complete-chain} offers a visualization of a $6$-chain in $\GF_8$. In this example, the $6$-chain forms a closed loop with $a_6 = a_0$. In fact, the following lemma shows that for any $a \in \poly(p,n)$ of degree at least $1$ and any two irreducible polynomials $f_1,f_2\in \GF_p[X]$ of degree $n$, the inverse chaining sequence of $a$ with respect to $(f_1,f_2)$ induces an even length \emph{cycle} (i.e., a closed loop). Hence forth, we refer to this closed loop sequence as the \emph{inverse chain} of $a$.

\begin{theorem}
Given any non-constant polynomial $a\in \GF_p[X]$ of degree at most $n-1$ and irreducible polynomials $f_1, f_2 \in \GF_p[X]$ of degree $n$, there exists an even integer $k$ such that the $k$-chain of $a$ with respect to $(f_1,f_2)$ is a cycle of length $k$. That is, $a_k = a_0 = a$ and $a_0,a_1,\ldots,a_{k-1}$ are all distinct polynomials.
Moreover, if $f_1 \neq f_2$ then the cycle/sequence has even length of $k\ge 4$. 
\end{theorem} 

\begin{proof}
When $f_1 = f_2$, clearly the sequence is a loop of length two since $ a_1 = \inv {a_0} {f_1} = \inv {a_0} {f_2}$ (thus, $a_0 = \inv{a_1} {f_2}$). 

Suppose $f_1 \neq f_2$. Consider the subgraphs on $\poly(p,n)\setminus\{0,1,\ldots,p-1\}$ induced by $\mathrm{inv}_{f_1}$ and $\mathrm{inv}_{f_2}$. These are two different \emph{perfect matchings} $\mathcal{M}_1$ and $\mathcal{M}_1$. By Claim~\ref{claim:distinct inverse}, we have that $\mathcal{M}_1\cap \mathcal{M}_2 = \emptyset$. Hence, the union of the two perfect matchings $\mathcal{M}_1\cup \mathcal{M}_2$ induces a set of disjoint even cycles on $\poly(p,n)\setminus\{0,1,\ldots,p-1\}$, of length at least $4$ .
Let $k$ be the size of the cycle containing $a$. The traversal of this cycle, starting from the edge $(a,\inv a {f_1})$ and ending with the edge $(\inv a {f_2}, a)$, is precisely the $k$-chain of $a$ with respect to $(f_1,f_2)$. 
\end{proof}

Observe that, by simply reversing the above traversal of the cycle containing $a$, we obtain the $k$-chain of $a$ with respect to $(f_2,f_1)$. That is, the inverse chain of $a$ with respect to $(f_1,f_2)$ is the reverse of the inverse chain of $a$ with respect to $(f_2,f_1)$. 

\begin{corollary}
Given any distinct irreducible polynomials $f_1,f_2 \in \GF_p[X]$ of degree $n$, inverse chaining induces a unique partition of $\poly(p,n)\setminus \{0,1,\ldots,p -1\}$ into cycles of even length $\ge 4$. 
\end{corollary}

Different pairs of irreducible polynomials produce distinct partition of cycles, since the matching induced by $\mathrm{inv}_f$, on non-constant polynomials, never shares an edge with the matching of $\mathrm{inv}_g$ when $f\neq g$ are different irreducible polynomials.

\subsection{Permutation Cycles}
In this section we explore how inverse chaining can be used to generate permutations from $\mathcal S_{q}$ or $\mathcal S_{m}$ for $m=p^n - p$. As mentioned above, the inverse chaining with respect to two distinct bases induces a partition of $\poly(p,n)\setminus \{0,1,\ldots,p_1\}$ into cycles. By fixing a orientation for each cycle, we view this as a cycle decomposition of a permutation.

\begin{theorem}\label{thm:permutation_cycles}
Given any two distinct irreducible polynomials $f_1,f_2 \in \GF_p[X]$ of degree $n$, consider the inverse chaining induced cycles over $\poly(p,n)\setminus \{0,1,\ldots,p_1\}$. For a fixed orientation of each cycle, we define a permutation $\sigma \in \mathcal{S}_q$ by considering these (directed) cycles as the cycle decomposition, and elements of $\GF_p$ are mapped to their inverse.\footnote{As a convention, $0$ maps to itself in the permutation.}
\end{theorem}

Observe that, no two different pairs of irreducible polynomials can induce the same permutation (for any choice of cycle orientation). Therefore, inverse chaining can be used to generate at least $M(p,n)\cdot(M(p,n)-1)$ many distinct permutations from $\mathcal S_q$, where the necklace polynomial $M(p,n)$ gives the number of monic irreducible polynomials in $\GF_p[X]$ of degree $n$. Clearly, not all $q!$ permutations in $\mathcal S_{q}$ (or $\mathcal S_{m}$) can be obtained by inverse chaining. For example, these permutation will never map an elements $a$ to itself, unless $a\in \{0,1,p-1\}$.

We note that chaining two involutions is completely different from a composition/product of two involutions. It is well known that any permutation can be obtained by a product of two involutions. In fact, there maybe multiple ways of factoring a permutation into a product of two involutions~\cite{petersen2013write,BurnetteCharles2017FPit}  (but these involutions are not limited to multiplicative inverse in finite fields). 

\section{Inverse Chaining with $\beta$ Bases}
In this section we consider inverse chaining under more than two bases. Given $\beta > 2$ distinct irreducible polynomials of degree $n$. We consider an inverse chain obtained by sequentially switching the basis with respect to the different irreducible polynomials.

\begin{definition}
    For any integer $k>0$, element $a\in \poly(p,n)$, and irreducible polynomial $f_1,f_2,\ldots,f_\beta \in \GF_p[X]$ of degree $n$, we define the $k$-chain of $a$ with respect to $(f_1,f_2,\ldots,f_\beta)$ as the sequence $(a_0,a_1,\ldots,a_k)$ such that
    \[
a_i = 
\begin{cases}
a \quad \text{if } i=0\\
\inv{a_{i-1}}{f_j} \quad \text{where } j = i \mod \beta
\end{cases}
\]
\end{definition}

We say that a $k$-chain of $a$ is a \emph{closed loop} if $k>0$ is the smallest integer such for any $\ell$-chain of $a$ with $\ell \ge k$ we have $(a_{k},\ldots,a_{l}) = (a_0,\ldots,a_{\ell - k})$. Recall that, in the special case of $\beta =2$, a closed loop $(a_0,\ldots,a_k)$ corresponded to simple cycle of length $k$ (i.e., $a_0,\ldots,a_{k-1}$ are distinct). However, this need not be the case for $\beta >0$, as the same element could appear in the closed loop multiple times; see supplementary materials for an example and Figure~\ref{fig:3-bases} for an illustration. Hence, the graph induced by the chaining under $\beta$ alternating basis is no longer just a union of disjoint cycles. In fact, an element can appear in multiple different closed loops when chaining under a fixed sequence of $\beta > 2$ alternating bases.

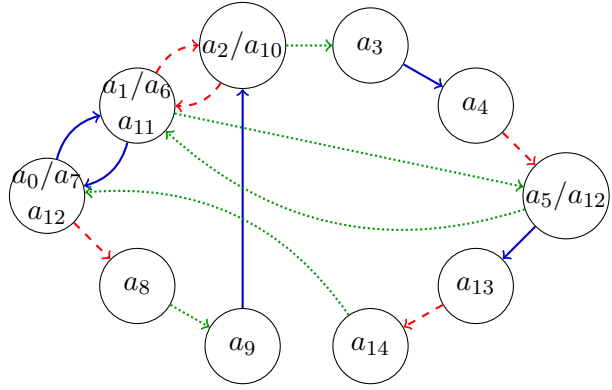
\begin{figure}
    \centering
    \begin{tikzpicture}
        % Define node style
\tikzset{circle node/.style={draw, circle, minimum size=1cm, inner sep=0pt}}
        % Draw the 4 circles with spacing
        {\node[circle node] (n0) at (0,0){};}
        \node[align = center, anchor = south] at (n0.south) {$a_{0} / a_7$ \\
        $a_{12}$};
        {\node[circle node] (n1) at (1.2,1.2){};}
        \node[align = center, anchor = south] at (n1.south) {$a_{1} / a_6$ \\
        $a_{11}$};
        {\node[circle node] (n2) at (1.2*2 + 0.2,1*2){$a_2 / a_{10}$};}
        {\node[circle node] (n3) at (1.2*3 + 0.7,1*2){$a_3$};}
        {\node[circle node] (n4) at (1.2*4 + 0.9,1.2){$a_4$};}
        {\node[circle node] (n5) at (1.2*5 + 0.9,0){$a_5 / a_{12}$};}
        {\node[circle node] (n8) at (1.2,-1.2){$a_8$};}
        {\node[circle node] (n9) at (1.2*2 + 0.2,-1*2){$a_9$};}
        {\node[circle node] (n14) at (1.2*3 + 0.7,-1*2){$a_{14}$};}
        {\node[circle node] (n13) at (1.2*4 + 0.9,-1.2){$a_{13}$};}
        % Draw blue edges
        \draw[bend left, ->, color = blue!80!black, thick] (n0) to node [auto]{} (n1);
        \draw[ ->, color = blue!80!black, thick] (n3) to node [auto]{} (n4);
        \draw[bend left, ->, color = blue!80!black, thick] (n1) to node [auto]{} (n0);
        \draw[ ->, color = blue!80!black, thick] (n9) to node [auto]{} (n2);
        \draw[ ->, color = blue!80!black, thick] (n5) to node [auto]{} (n13);
        %Draw red edges
        \draw[ bend left,->, color = red, thick, dashed] (n1) to node [auto]{} (n2);
        \draw[ ->, color = red, thick, dashed] (n4) to node [auto]{} (n5);
        \draw[->, color = red, thick, dashed] (n0) to node [auto]{} (n8);
        \draw[ bend left, ->, color = red, thick, dashed] (n2) to node [auto]{} (n1);
        \draw[ ->, color = red, thick, dashed] (n13) to node [auto]{} (n14);
        %Draw green edges
        \draw[ ->, color = green!60!black, thick, densely dotted] (n2) to node [auto]{} (n3);
        \draw[ bend left, ->, color = green!60!black, thick, densely dotted] (n5) to node [auto]{} (n1);
        \draw[ ->, color = green!60!black, thick, densely dotted] (n8) to node [auto]{} (n9);
        \draw[ ->, color = green!60!black, thick, densely dotted] (n1) to node [auto]{} (n5);
        \draw[bend right, ->, color = green!60!black, thick, densely dotted] (n14) to node [auto]{} (n0);
    \end{tikzpicture}
    \caption{An illustrative example of a $15$-chain closed loop in $\GF_{16}$ with three alternating bases (in blue, red, green order) starting with $a_0 = x^2+x+1$. The blue, red, and green edges denote multiplicative inverse operation with three irreducible polynomials, $x^4+x+1, x^4+x^3+1$, and $x^4+x^3+x^2+x+1$ respectively.}
    \label{fig:3-bases}
\end{figure}

\begin{lemma}
Given any non-constant polynomial $a\in \GF_p[X]$ of degree at most $n-1$ and $\beta$ distinct irreducible polynomials of degree $n$. The minimum length of closed loop is at least $\beta$.
\end{lemma}
\begin{proof}
By definition, a $k$-chain of $a$ forms a closed loop only if $a_{k} = a_0$ and $a_{k+1} = a_1$. Recall that the inverse of $a$ under different basis are all distinct (Claim~\ref{claim:distinct inverse}). Since, $a_{k+1} = \inv{a_k}{f_1}$, we have $k+1 = 1 \mod \beta$. That is, $k = c\cdot\beta$ for some integer $c>0$.
\end{proof}

Moreover, unlike alternating with only $2$ bases, for $\beta>2$ the minimum length of a closed loop can be $\beta$. For example, in $\GF_{2^6}$ there is a closed loop length $3$ under inverse chaining with $\beta = 3$ irreducible polynomials; see the supplementary material.

\section{Supplementary Material}
The colab notebook hosted \href{https://colab.research.google.com/drive/120hUmDyI__0we_8uhRP8r5VVX8KwHK9l?usp=sharing}{\underline{here}}\footnote{\url{https://colab.research.google.com/drive/120hUmDyI__0we_8uhRP8r5VVX8KwHK9l?usp=sharing}} implements various examples of inverse chaining in finite fields.

\section{Discussion and Open Problems}
The preliminary numerical exploration (included as supplementary material) raises many questions about the properties of the graph induced by inverse chaining on \emph{average} and/or for \emph{large} $n$. For example, under inverse chaining w.r.t.\ two irreducible polynomials the length of a closed loop can be as small as $4$ in the worst case, however it is possible to cover {all non-constant elements} in a   \emph{single closed loop}. Indeed, there exists pairs of irreducible polynomials in $\GF_2[X]$ of degree $6$ that produce a closed loop of length $4$, while Figure~\ref{fig:complete-chain} shows an example covering {all non-constant elements} of $\poly(2,3)$. So if we consider two \emph{random} irreducible polynomials in $\GF_2[X]$ (of degree $n$), can we answer the following questions about the \emph{random inverse chaining} induced: How does the minimum/maximum/average length of a closed loop grow asymptotically in $n$? Similarly, what is the (expected) number of cycles/closed-loops in a random inverse chaining partition? Can we characterize when a single closed loop that spans all of $\poly(2,n)\setminus \{0,1\}$ is produced? Further, it would also be interesting to understand the distributional properties of permutations induced by random inverse chaining.

\bibliographystyle{plain}
\bibliography{bib}

\begin{thebibliography}{1}

\bibitem{BurnetteCharles2017FPit}
Charles Burnette.
\newblock {\em Factoring Permutations into the Product of Two Involutions: A
  Probabilistic, Combinatorial, and Analytic Approach}.
\newblock PhD thesis, Drexel University, 2017.

\bibitem{Conrad_Roots}
Keith Conrad.
\newblock Roots and irreducibles.
\newblock
  \url{https://kconrad.math.uconn.edu/blurbs/galoistheory/rootirred.pdf}.

\bibitem{petersen2013write}
T~Kyle Petersen and Bridget~Eileen Tenner.
\newblock How to write a permutation as a product of involutions (and why you
  might care).
\newblock {\em arXiv preprint arXiv:1202.5319}, 2012.

\end{thebibliography}

\end{document}